\documentclass[12pt]{article}
\usepackage{amsmath,amssymb,amsthm, amsfonts}
\usepackage{verbatim, enumerate}
\usepackage{hyperref}
\usepackage{graphicx}
\usepackage{url}
\usepackage[mathlines]{lineno}
\usepackage{dsfont} 
\usepackage{tikz, graphicx, subcaption, caption} 
\usepackage[algo2e,ruled,vlined]{algorithm2e}
\usepackage{color}
\definecolor{red}{rgb}{1,0,0}

\definecolor{blue}{rgb}{0,0,1}

\definecolor{green}{rgb}{0,.6,.4}

\numberwithin{figure}{section}

\newtheorem{thm}{Theorem}[section]
\newtheorem{cor}[thm]{Corollary}
\newtheorem{lem}[thm]{Lemma}
\newtheorem{prop}[thm]{Proposition}

\newtheorem{obs}[thm]{Observation}

\newtheorem{quest}[thm]{Question}

\theoremstyle{definition}
\newtheorem{rem}[thm]{Remark}

\theoremstyle{definition}

\theoremstyle{definition}
\newtheorem{ex}[thm]{Example}

\def\mtx#1{\begin{bmatrix} #1 \end{bmatrix}}

\def\Pk#1{P^{(#1)}}

\newcommand{\bq}{{\bf q}}

\newcommand{\pr}{\mathbf{Pr}}
\newcommand{\e}{\mathbf{E}}
\newcommand{\var}{\mathbf{Var}}
\newcommand{\ept}{\operatorname{ept}}
\newcommand{\rad}{\operatorname{rad}}

\newcommand{\Z}{\operatorname{Z}}

\newcommand{\thr}{\operatorname{th}}
\newcommand{\thpf}{\operatorname{th}_{\rm pzf}}
\newcommand{\thp}{\operatorname{th}_+}

\newcommand{\pt}{\operatorname{pt}}
\newcommand{\ptpf}{\pt_{\rm pzf}} 
\newcommand{\ptp}{\operatorname{pt}_+}

\newcommand{\dist}{\operatorname{dist}}

\newcommand{\bit}{\begin{itemize}}
\newcommand{\eit}{\end{itemize}}
\newcommand{\ben}{\begin{enumerate}}
\newcommand{\een}{\end{enumerate}}
\newcommand{\beq}{\begin{equation}}
\newcommand{\eeq}{\end{equation}}
\newcommand{\bea}{\begin{eqnarray*}} 
\newcommand{\eea}{\end{eqnarray*}}
\newcommand{\bpf}{\begin{proof}}
\newcommand{\epf}{\end{proof}\ms}
\newcommand{\bmt}{\begin{bmatrix}}
\newcommand{\emt}{\end{bmatrix}}
\newcommand{\ms}{\medskip}

\newcommand{\lc}{\left\lceil}
\newcommand{\rc}{\right\rceil}
\newcommand{\lf}{\left\lfloor}
\newcommand{\rf}{\right\rfloor}
\newcommand{\lp}{\!\left(}
\newcommand{\rp}{\right)}

\newcommand{\noi}{\noindent}
\newcommand{\ol}{\overline}

\title{Propagation time for probabilistic zero forcing}
\author{Jesse Geneson\thanks{Department of Mathematics, Iowa State University,
Ames, IA 50011, USA (geneson@iastate.edu).}\and Leslie Hogben\thanks{Department of Mathematics, Iowa State University,
Ames, IA 50011, USA (hogben@iastate.edu) and American Institute of Mathematics, 600 E. Brokaw Road, San Jose, CA 95112, USA (hogben@aimath.org).}}

\begin{document}


\maketitle\vspace{-10pt}

 \begin{abstract} Zero forcing is a coloring game played on a graph that was introduced more than ten years ago  in several different applications.  The goal is to color all the vertices blue by repeated use of a (deterministic) color change rule.  Probabilistic zero forcing was introduced by Kang and Yi in [Probabilistic zero forcing in graphs, {\em Bull. Inst. Combin. Appl.} 67 (2013), 9--16] and yields a discrete dynamical system, which is a better model for some applications.  Since in a connected graph any one vertex can eventually color the entire graph blue using probabilistic zero forcing, the expected time to do this is a natural parameter to study. We determine expected propagation time exactly for paths and cycles, establish the asymptotic value for stars, and present asymptotic  upper and lower bounds for any graph in terms of its radius and  order.   We apply these results to obtain values and bounds on $\ell$-round probabilistic  zero forcing, throttling number for probabilistic zero forcing, and confidence levels for propagation time.   \end{abstract}

\noi {\bf Keywords} probabilistic zero forcing, expected propagation time, $\ell$-round probability, confidence propagation, throttling

\noi{\bf AMS subject classification} 05D40, 05C57, 05C15, 05C12


\section{Introduction}\label{s:intro}

Zero forcing is a coloring process on a graph that  was introduced independently in the study of control of quantum systems in mathematical physics \cite{graphinfect} and the maximum nullity problem in combinatorial matrix theory \cite{AIM08}. It was later observed to have connections with graph searching \cite{Y13} and power domination \cite{PDZ}.  Variants of zero forcing  and derived parameters such as propagation time and throttling have also been studied (see, for example, \cite{AIM17}). 

Zero forcing on a graph $G$ is described by the following (standard) zero forcing color change rule: Given a set $B$ of vertices of $G$ that are colored blue with the remaining vertices colored white, a blue vertex $u$ can change the color of (force) a white vertex $w$ to blue if $w$ is the only white neighbor of $u$; this is denoted by $u\to w$.  A zero forcing set of $G$ is a set $Z\subseteq V(G)$ of vertices such that when the vertices of $Z$ are colored blue and the remaining vertices are colored white, every vertex can eventually be colored blue by repeated applications of the color change rule. The zero forcing number of $G$, $\Z(G)$, is the minimum cardinality of a zero forcing set.  Throughout this paper the terms {\em zero forcing}, {\em zero forcing color change rule}, and {\em zero forcing set} will refer to  the definitions just given; a force performed using the zero forcing color change rule is also called a {\em deterministic force}.

Probabilistic zero forcing was introduced by Kang and Yi in \cite[Definition 1.1]{KY13}.  
Given a  set $B$ of currently blue vertices, in one round each blue vertex $u\in B$ {\em fires} at, i.e., attempts to force (change the color to blue), each of its white neighbors $w\in \ol B$ independently with  probability   
\beq\label{eq:pforce}\pr(u\to w)=\frac {|N[u]\cap B|}{\deg u}.\eeq The coloring rule just described is    
the {\em probabilistic color change rule} (in \cite{KY13} $\pr(u\to w)$ is denoted by $F(u\to w)$).  {\em Probabilistic zero forcing} refers to the process of coloring a graph blue by repeatedly applying the probabilistic color change rule.
As noted in \cite{KY13},  the definition of probability of a force in \eqref{eq:pforce} has the property that a deterministic force will be performed with probability one.  

The probabilistic color change rule produces a discrete dynamical system that plausibly describes many applications.  For example, zero forcing is sometimes used to model rumor spreading in social networks, and given human nature a probabilistic model is more realistic. A probabilistic model is also more realistic for the spread of infection among a population, or the spread of a computer virus in a network.

Probabilistic zero forcing also presents an entirely new perspective on  graph coloring, and it is necessary to revise the parameters of interest. 
For zero forcing, determining the minimum number of vertices needed to color the graph blue 
is a main research question (on a connected graph of order at least three that is not a path, no one vertex is a zero forcing set). 
However,  in probabilistic zero forcing {\em any} one vertex in a connected graph can eventually force the entire graph blue.
Since a minimum zero forcing set is not of interest, it is natural to ask what parameter(s) are most interesting to study for probabilistic zero forcing. 

Kang and Yi define the probability $P_A(G)$ as follows: Let $k_o$ be the first round in which it is possible to have a deterministic zero forcing set colored blue, starting with exactly the vertices in $A$ colored blue. Define $P_A(G)$ to be the probability that a deterministic zero forcing set has been colored blue in round $k_o$. We discuss this parameter in Section \ref{s:PAG}, where we present a counterexample to one of the properties claimed for it in \cite{KY13}.    

A more natural object of study is the expected number of rounds needed to color all vertices blue with a given starting set of vertices, especially starting with a single vertex; this is the parameter studied in Section \ref{s:ept}.
In the deterministic case, the {\em propagation time} $\pt(G,Z)$  of a zero forcing set $Z$  for $G$ is the number of time steps needed to color all vertices blue, performing independent forces simultaneously at each time step. The {\em propagation time} $\pt(G)$ of a  graph $G$ is the minimum of $\pt(G,Z)$ over all minimum zero forcing sets $Z$. 

 We can recast the definition of zero forcing in parallel with that of probabilistic zero forcing, which is particularly useful for defining a time step in the study of propagation time.
Given a  set $B$ of currently blue vertices, in one time step (analogous to a round) each blue vertex $u\in B$ {fires} at, i.e., attempts to force (change the color to blue), each of its white neighbors $w\in \ol B$ independently with  probability   
\beq\label{eq:force}\pr(u\to w)=\begin{cases} 
1 & \mbox{if $w$ is the only white neighbor of $u$}\\
0 &  \mbox{otherwise}
\end{cases}.\eeq

The {\em probabilistic  propagation time} of a nonempty set  $Z$ of vertices of a connected graph $G$, $\ptpf(G,Z)$, is a random variable that reflects the time (number of the round) at which the last white vertex turns blue when applying a probabilistic zero forcing process starting with the set $Z$ blue (if $G$ is not connected, we assume $Z$ contains at least one vertex from each connected component of $G$). 
For a graph $G$ of order $n$ and a set $Z\subseteq V(G)$ of vertices, the {\em expected propagation time of $Z$ for $G$}  is the expected value of the  propagation time of $Z$, i.e., 
\[\ept(G,Z)=\e[\ptpf(G,Z)].\]
  The {\em expected propagation time} of a connected graph $G$ is the minimum of the expected propagation time of $Z$ for $G$ over all one vertex sets $Z$, i.e., 
\[\ept(G)=\min\{\ept(G,\{v\}):v\in V(G)\}.\]
In Section \ref{s:ept} we determine $\ept(G)$ exactly when $G=P_n$ (a path on $n$ vertices) or $G=C_n$ (a cycle on $n$ vertices).  We also determine asymptotic upper and lower bounds on expected propagation time and apply them to additional families of graphs.
 
Aazami introduced the study of {\em $\ell$-round zero forcing} in \cite{Aaz08, Aaz10} (an $\ell$-round zero forcing set can force the entire graph blue in at most $\ell$ time steps).   In Section \ref{s:ell-round} we define $\ell$-round probabilistic zero forcing as the probability of all vertices being colored blue in $\ell$ rounds; this parameter is a possible alternative to  $P_A(G)$,  the parameter introduced by Kang and Yi in \cite{KY13}.
In Section \ref{s:alpha} we approach propagation time from the perspective of confidence level, such as what is the minimum number of rounds need to ensure the entire graph will be blue with probability at least $\alpha$ (e.g., $\alpha=.95$ for a 95\% confidence level).
Many of the results on expected propagation time can be applied to obtain results for  these parameters.

Another parameter of interest in the study of zero forcing and other graph search parameters is throttling number.   Throttling, initially considered in \cite{BY13}, refers to considering a combination of the resources used to accomplish a task and time needed to accomplish the task and has been studied for variants of zero forcing and Cops and Robbers on graphs. For a zero forcing set $Z$ of $G$, $\thr(G,Z)=|Z|+\pt(G,Z)$, and the 
{\em throttling number} of $G$ is $\thr(G)=\min\{\thr(G,Z)\}$.  {\em Probabilistic throttling} is defined using expected propagation time:
\[\thpf(G,Z)=|Z|+\ept(G,Z)\ \mbox{ and }\ \thpf(G)=\min\{\thpf(G,Z)\};\] 
$\thpf(G)$ is called the {\em probabilistic throttling number} and is discussed in Section \ref{s:throt}. Throttling for confidence propagation time is discussed in Section \ref{s:alphath}.  Many of the results on expected propagation time can be applied to the throttling parameters.

We conclude this introduction with some notation that will be used throughout and statements of results from probability theory that will be used repeatedly. The distance between vertices $u$ and $v$ is denoted by $\dist(u,v)$, and $\dist(u,S)=\min_{x\in S}\dist(u,x)$ for $S\subseteq V(G)$.  For   functions $f(n)$ and $g(n)$ from the nonnegative or positive integers to the real numbers,  asymptotic bounds are defined as follows:  $f(n)=o(g(n))$ if $\lim_{n\to\infty} \frac{f(n)}{g(n)}=0$,   $f(n)=O(g(n))$ if there exists $c>0$ such that $f(n)\le cg(n)$ for all $n$ sufficiently large, $f(n)=\omega(g(n))$ if $g(n)=o(f(n))$, $f(n)=\Omega(g(n))$ if $g(n)=O(f(n))$, and  $f(n)=\Theta(g(n))$ if $f(n)=O(g(n))$ and $f(n)=\Omega(g(n))$.

\begin{thm}[Markov's inequality] Let $X$ be a nonnegative random variable.  For any constant $a>0$,
\[\pr(X\ge a)\le \frac{\e[X]}{a}.\]
\end{thm}

\begin{thm}[Chebyshev's inequality] Let $X$ be a random variable.  For any constant $c>0$,
\[\pr(|X-\e[X]|\ge c)\le \frac{\var(X)}{c^2}.\]
\end{thm}

\begin{obs}\label{o:first} If the probability of an event is   $p$, then the expected trial of the event's first occurrence in repeated trials  is  $\frac 1 p$.
\end{obs}


\section{Expected propagation time}\label{s:ept}

In this section we determine the expected propagation time 
for several families of graphs.  We also develop several tools for bounding   the expected propagation time.

\begin{prop}\label{prop:ept-cycle} For a cycle of order $n > 2$, 
\[\ept(C_{n}) = \begin{cases} 
\frac n 2 +\frac 1 3 & \mbox{if $n$ is even}\\
\frac n 2 +\frac 1 2 &  \mbox{if $n$ is odd}
\end{cases}.\]
\end{prop}

\begin{proof}
Observe that $\ept(C_{n})$ is the sum of the expected number of rounds until the first successful probabilistic force plus the number of rounds for the remainder of the vertices to be deterministically forced blue, since the process becomes deterministic as soon as there are at least two adjacent blue vertices.    Since the probability that one blue vertex forces at least one of its white neighbors in any round is $\frac 3 4$, the expectation for the first force is  $\frac 4 3$ by Observation \ref{o:first}.  The number of rounds $r$ needed to deterministically force all remaining vertices once two or three consecutive vertices are blue is the maximum of the distance $\dist(w,B)$  
of a white vertex $w$ to the set $B$ of (two or three) blue vertices. For $n$ even, $r=\frac{n-2}2$ (regardless of whether there are two or three blue vertices), so  $\ept(C_{n}) =\frac 4 3+\frac{n-2}2= \frac n 2 +\frac 1 3$. 

Now assume $n$ is odd. The case of three blue vertices must be distinguished from two blue vertices because it affects $r=\max_{w\in V(G)\setminus B}\dist(w,B)$.  When there are three consecutive blue vertices, $r=\frac{n-3}2$, whereas with two adjacent blue vertices $r=\lc\frac{n-2}2\rc=\frac{n-1}2$.  Assuming that  the first  force has taken place in the prior round, the probability of exactly three blue vertices (two forces occurred) is $\frac 1 3$ and the probability of exactly two blue vertices (one force occurred) is $\frac 2 3$.  Thus  $\ept(C_{n})=\frac 4 3+ \frac 1 3 (\frac{n-3}{2}) + \frac 2 3 (\frac{n-1}2)  =\frac n 2 +\frac 1 2$.
\end{proof}

\begin{prop}\label{prop:ept-path} For a path of order $n > 2$,
\[\ept(P_{n}) = \begin{cases} 
 \frac n 2+\frac 2 3 & \mbox{if $n$ is even}\\
\frac n 2 +\frac 1 2 &  \mbox{if $n$ is odd}
\end{cases}.\]
\end{prop}

\begin{proof} 
The proof is similar to that of Lemma \ref{prop:ept-cycle}.  
{Number the vertices of $P_n$ in order from 1 to $n$ and let $u=\lc\frac n 2\rc$ be the blue vertex.}  For odd $n$, the situation is the same as that of a cycle and $\ept(P_n)=\frac 4 3 + \frac 1 3 (\frac{n-3}{2}) + \frac 2 3 (\frac{n-1}2) = \frac n 2 +\frac 1 2$.    

Now assume $n$ is even, which means {the distance  from  $u=\frac n 2$ to  $n$  is  $\frac n 2$ whereas  the distance from $u$  to $1$  is  $\frac {n} 2-1$.   
Assuming that  at least one force takes place, the probability of $\frac n 2 +1$ being forced (with or without $\frac n 2 -1$ being forced) is $\frac 2 3$ and the probability of only $\frac n 2 -1$ being forced  is $\frac 1 3$}.  Thus  $\ept(P_{n}) = \frac 4 3 + \frac 2 3 (\frac{n}{2}-1) + \frac 1 3 (\frac{n}2) = \frac n 2 +\frac 2 3$.
\end{proof}

 A probabilistic zero forcing process is a Markov chain; see  \cite{KY13}   for a discussion of the construction of the Markov transition matrix, or see the proof of Lemma \ref{lem:starlow} for an actual such matrix.  In the next remark, we explain how to  use the Markov transition matrix $M$ to compute the expected propagation time $\ept(G,Z)$.  
 
 \begin{rem}\label{rem:Markov} {\rm Given a graph $G$ and an initial set of blue vertices $Z$, each possible set of blue vertices in the probabilistic zero forcing process is a state.  Suppose there are $m$ states, in state 1 the blue vertices are exactly those in $Z$, and in state $m$ all vertices are blue.  Let $M$ denote the Markov transition matrix and $\bq=[1,0,\dots,0]$. Then the probability that all vertices are blue after round  $r$ is $\left(\bq M^r\right)_m=\lp M^r\rp_{1m}$, so
\[\ept(G,Z)=\sum_{r=1}^\infty r\left(\lp M^r\rp_{1m}-\lp M^{r-1}\rp_{1m}\right).\]
}\end{rem}
The method described in Remark \ref{rem:Markov} is particularly  useful for specific small cases, as in the next lemma.

\begin{lem}\label{lem:starlow} Let $G$ be a 
graph and let  $v$ be a  vertex of $G$.  Then
\[\ept(G[N[v]])\le \begin{cases} 
1 & \mbox{if }\deg v =1\\
2 &  \mbox{if }\deg v=2\\
2.76316  &  \mbox{if }\deg v = 3\\
 3.34171 &  \mbox{if }\deg v = 4
\end{cases}.\]
\end{lem}
\bpf  Let $d=\deg v$.  Let $G'$ be the graph obtained from $G[N[v]]$ by removing all edges except for those adjacent to $v$, so $G'$ is a star. Since the degree of $v$ is the same in both $G$ and $G'$ and additional forcing of vertices in $G[N[v]]$ may be possible in $G$, $\ept(G[N[v]])\le  \ept(G',\{v\})$.  Thus it suffices to determine $\ept(G',\{v\})$ for $d=1,2,3,4$.

The case $\deg v=1$ is deterministic zero forcing.  
For $\deg v=2$, $\ept(G',\{v\})=\ept(P_3)=2$ by Proposition \ref{prop:ept-cycle}.
The remaining probabilities are computed using the Markov transition matrices $M_d$ for $d=3,4$ with states $0, 1,\dots,d$ blue leaves where
\[M_3=\mtx{\frac{8}{27} & \frac{4}{9} & \frac{2}{9} & \frac{1}{27}
   \\[2pt]
 0 & \frac{1}{9} & \frac{4}{9} & \frac{4}{9} \\[2pt]
 0 & 0 & 0 & 1 \\
 0 & 0 & 0 & 1}\qquad\mbox{and}\qquad M_4=\mtx{ \frac{81}{256} & \frac{27}{64} & \frac{27}{128} &
   \frac{3}{64} & \frac{1}{256} \\[2pt]
 0 & \frac{1}{8} & \frac{3}{8} & \frac{3}{8} &
   \frac{1}{8} \\[2pt]
 0 & 0 & \frac{1}{16} & \frac{3}{8} & \frac{9}{16} \\[2pt]
 0 & 0 & 0 & 0 & 1 \\
 0 & 0 & 0 & 0 & 1}\!\!;\] see \cite{markov-star} for the computational details. \epf

Next we prove  a main lemma that provides an upper bound for expected propagation time for the neighborhood of a vertex based on its degree.
\begin{lem}\label{lem:starlog} Let $G$ be a 
graph.  Then for any vertex $v$ of $G$,
\[\ept(G[N[v]]) = O(\log{\deg v}).\]
\end{lem}

\begin{proof} Let $d=\deg v$.  
As in the proof of Lemma \ref{lem:starlow}, let $G'$ be the graph obtained from $G[N[v]]$ by removing all edges except for those adjacent to $v$; note that the order of $G'$ is $d+1$.  
 It suffices to prove that $\ept(G',\{v\}) = O(\log d)$. Since asymptotic bounds are for sufficiently large values, we assume $d\ge 272$.  We establish the following three claims, using $b$ to denote the number of currently blue vertices in $G'$ and $w=d+1-b$ to denote the number of currently white vertices.  
 \ben[(C1)]
 \item\label{1b4} For $1\le b \le 4$, the probability of at least one new blue vertex in $G'$ in one  round is at least $\frac 1 2$.
 \item\label{4bd2} For $4 \le b \le \frac d 2$,  the probability  of at least $\frac b 4$ new blue vertices in $G'$ in one  round is at least $\frac 1 3$.
\item\label{d2bd} For $1 \le w \le \frac d 2$,  the probability of at least $\frac w 4$ new blue vertices in $G'$ in one  round is at least $\frac {16}{17}$.\een
Once the three claims have been established, by Observation \ref{o:first} the expected number of rounds to satisfy the condition for one new blue vertex, at least $\frac b 4$ new blue vertices, or at least $\frac w 4$ new blue vertices,  is at most $2$, $3$, or $\frac{17}{16}$, respectively.
Thus the expected number of rounds to go from one to at least four blue vertices is $O(1)$.  Starting with between 4 and $\frac d 2$ blue vertices, the expected number of rounds until the number of blue vertices goes up by 25\% is at most $3$.  Thus the expected number of rounds until the number of blue vertices is at least $\lp\frac 5 4\rp^r 4 $ is at most $3r$, and  the expected number of rounds to go from at least four blue vertices to at least $\frac d 2 + 1$ blue vertices is $O(\log d)$.    
Starting with at least $\frac d 2 + 1$ blue vertices, or at most $\frac d 2$ white vertices, the expected number of rounds until the number of white vertices decreases  by 25\% is  $\frac{17}{16}$.  Thus the expected number of rounds until the number of white vertices is at most $\lp\frac 3 4\rp^r \lp\frac d 2 \rp $ is at most $\frac{17}{16}r$, and   the expected number of rounds to go from   at least $\frac d 2 + 1$ blue vertices  to all blue vertices is $O(\log d)$.  
  Thus $\ept(G',\{v\}) = O(\log{d})$.
 
For (C\ref{1b4}), note that when there are $b$ blue vertices, the probability that at least one additional vertex gets colored blue in the current round is 
\[1-(1-\frac b d)^{d+1-b} \geq 1-(1-\frac b d)^{d-b} \geq 1-\frac{1}{e^{b (d-b)/d}} \geq \frac 1 2\] for $b \leq 4$. 

For (C\ref{4bd2}), let $p(b)$ be the probability that the number of vertices forced  in the current round is at least $\frac b 4$, given that there are currently $b$ blue vertices and $4\le b\le \frac d 2$.   For each white vertex $v_{1}, \dots, v_{d+1-b}$, define $X_{i}$ to be $1$ if $v_{i}$ is colored blue in this round and $0$ otherwise. Let $X = \sum_{i = 1}^{d+1-b} X_{i}$. Since the $X_i$'s are independent identically distributed (i.i.d.)  with $\e[X_{i}] = \frac b d$ and $\var[X_{i}] = \frac {b(d-b)}{ d^2}$, we have $\e[X] = \frac {b(d+1-b)}d>\frac b 2$ and $\var[X] = \frac{b(d-b)(d+1-b)}{d^2}\le b$. 
We consider two subcases: $36 \leq b \leq \frac d 2$ and $4 \leq b \le 35$. 
For $36 \leq b \leq \frac d 2$,
 \bea 1-p(b)=\pr\lp X  < \frac b 4\rp \le \pr\lp X \le \frac b 4\rp  &=& \pr\lp \frac b 2-X \geq \frac b 4\rp  \leq \pr\lp \e[X] - X \geq \frac b 4\rp\\
 & \leq& \pr\lp |X-\e[X]| \geq \frac b 4\rp \leq \frac{\var[X]}{(\frac b 4)^{2}} \leq \frac{16}{b} < \frac 1 2\eea 
where $\pr\lp |X-\e[X]| \geq \frac b 4\rp \leq \frac{\var[X]}{(\frac b 4)^{2}}$ is Chebyshev's inequality and the other equalities and inequalities follow from the definitions, the values above, or algebraic manipulation.  Thus in this case $p(b)> \frac 1 2> \frac 1 3$.
 For $4 \leq b \le 35$ note first that $d\ge 272$ implies $\frac {d+1-b}d\ge \frac 7 8$. 
 \bea\pr\lp X \leq \frac b 4\rp = \pr\lp \frac{7b}8-X \geq \frac{5b}8\rp &\leq& \pr\lp \e[X] - X  \geq \frac{5b}8\rp  \\
 &\leq& \pr\lp |X-\e[X]| \geq \frac{5b}8\rp  \leq \frac{\var[X]}{(\frac{5b}8)^{2}} \leq \frac{64}{25b} < \frac 2 3.\eea
Thus in this case $p(b)> \frac 1 3$.

For (C\ref{d2bd}), let $q(w)$ be the probability that the number of new blue vertices in the current round is at least $\frac w 4$, given that there are currently $w \le \frac d 2$ white vertices in $G'$. 
For each white vertex $v_{1}, \dots, v_{w}$, define $Y_{i}$ to be $1$ if $v_{i}$ is colored blue and $0$ otherwise. Let $Y = \sum_{i = 1}^{w} Y_{i}$. Since the $Y_i$'s are i.i.d. with $\e[Y_{i}] = \frac{d+1-w}d$ and $\var[Y_{i}] = \frac{(d+1-w)(w-1)}{d}$, we have $\e[Y] = \frac{(d+1-w)w}{d}\ge \frac w 2$ and $\var[Y] = \frac{(d+1-w) (w-1)w}{d^2}$. Since $d\ge 272$,
\bea 1-q(w)\le \pr\lp Y \leq \frac w 4\rp  &=& \pr\lp \frac w 2-Y \geq \frac w 4\rp \  \leq\ \pr\lp \e[Y] - Y\geq \frac w 4\rp  \\
&\leq& \pr\lp |Y-\e[Y]| \geq \frac w 4\rp  \leq \frac{\var[Y]}{(\frac w 4)^{2}} \leq \frac{16}{d} \le \frac 1 {17}. \qedhere
\eea 
%
\end{proof}

\begin{cor}\label{c:univers}
If a graph $G$ of order $n$ has a universal vertex, then $\ept(G) = O(\log{n})$. 
\end{cor}

\begin{thm}\label{starlower}
For the star on $n+1$ vertices, $\ept(K_{1, n}) = \Theta(\log{n})$.
\end{thm}
\begin{proof}
The upper bound follows from   Lemma \ref{lem:starlog}. For the lower bound, 
let $h(b)$ be the probability that the number of new blue vertices colored in the current round is at most $4b$, given that there are currently $b$ blue vertices in $K_{1, n}$ and  the center vertex is blue.
Using the same setup with the random variables $X_{i}$ for each $i = 1, \dots, n+1-b$ and  $X=\sum_{i=1}^{n+1-b} X_i$  as in Lemma \ref{lem:starlog}, $\e[X] = \frac {b(n+1-b)}n\le b$ and $\var[X] = \frac{b(n-b)(n+1-b)}{n^2}\le b$.
We again use Chebyshev's inequality to show that $h(b) = 1-O({\sqrt{n}}^{-1})$ for $\sqrt{n} \leq b \leq \frac n 2$:  \vspace{-3pt}
\[ 1-h(b) \le \pr\lp X - b \geq 3b\rp   \leq  \pr\lp |X-\e[X]| \geq 3b\rp  \leq \frac{\var[X]}{(3b)^{2}} \leq \frac{1}{9b} \le \frac 1 {9\sqrt n}= O\left(\frac 1 {\sqrt n}\right)\!.\vspace{-3pt}\]


 Since starting with $\sqrt{n} \leq b \leq \frac n 2$ blue vertices and coloring at most $4b$ additional vertices blue means there are at most $5b$ blue vertices after the round, the probability that there are at most $5^rb$ blue vertices after $r$ rounds is at least $\left(1-O(\frac 1 {\sqrt{n}})\right)^r$\!.   Thus going from $b\le \sqrt n$ blue vertices to at least $\frac n 2$ blue vertices requires that $5^r\sqrt n\ge  \frac n 2$, or $r\ge  \log_{5}\!\left(\frac {\sqrt n} 2\right)$.   
Thus  with probability at least $(1-O({\sqrt{n}}^{-1}))^{\log_{5}(\sqrt n/2)} = 1-o(1)$, it takes at least $\log_{5}\left(\frac {\sqrt n}2\right)$ rounds for the number of blue vertices to increase from at most $\sqrt{n}$ to at least $\frac n 2$. Since $\sqrt{n} \geq 2$ for $n > 3$, we have covered the case in which the first blue vertex is a leaf rather than the center, because in that case the expected propagation time is one more than the expected propagation time starting with two blue vertices, one of which is the center.   Thus 
 $\ept(K_{1,n})=\Omega(\log{n})$.
\end{proof}

\begin{prop}\label{klower}
For the complete graph on $n$ vertices, $\ept(K_{n}) = \Omega(\log{\log{n}})$.
\end{prop}

\begin{proof}
Let $\hat h(b)$ be the probability that the number of additional blue vertices colored in the current round is at most $4 b^2$, given that there are currently $b$ blue vertices.
For each white vertex $v_{1}, \dots, v_{n-b}$, define $X_{i}$ to be $1$ if $v_{i}$ gets colored blue and $0$ otherwise and $X = \sum_{i = 1}^{n-b} X_{i}$. Since the $X_i$'s are i.i.d. with $\e[X_{i}] = 1-\lp1-\frac b {n-1}\rp^b$ and $\var[X_{i}] = \lp1-\lp1-\frac b {n-1}\rp^b\rp\lp1-\frac b {n-1}\rp^b$, 
we have $\e[X] = \lp1-\lp1-\frac b {n-1}\rp^b\rp(n-b) $, and furthermore 
$\var[X] =  \lp1-\lp1-\frac b {n-1}\rp^b\rp\lp1-\frac b {n-1}\rp^b (n-b) \le \e[X] \leq b^2$ by Bernoulli's inequality.  By algebraic manipulation and Chebyshev's inequality,
 \[ 1-\hat h(b)  \le \pr\lp X - b^2 \geq 3b^2\rp   \leq \pr\lp |X-\e[X]| \geq 3b^2\rp 
  \leq \frac{\var[X]}{(3b^2)^{2}} \leq \frac{1}{9b^2} = O\lp\frac 1{b^2}\rp\]
 Since starting with $\log{n} \leq b \leq n$ blue vertices and coloring at most $4b^2$ additional vertices blue means there are at most $5b^2\le b^3$ blue vertices after the round for $b \ge 5$, 
 the probability that there are at most $b^{(3^r)}$ blue vertices after $r$ rounds is at least $\left(1-O(\frac 1 {(\log{n})^2})\right)^r$\!.   Thus going from $b\le \log{n}$ blue vertices to $n$ blue vertices requires that $(\log{n})^{(3^r)} \ge  n$, or $r\ge  \log_{3}\!\left(\frac {\log{n}} {\log{\log{n}}}\right)$.  
Thus with probability at least $1-o(1)$, it takes $\Omega(\log{\log{n}})$ rounds for the number of blue vertices to increase from at most $\log{n}$ to exactly $n$. 
 Therefore 
 $\ept(K_{n}) = \Omega(\log{\log{n}})$.\end{proof}

The next result is immediate from Lemma \ref{lem:starlog} and Proposition \ref{klower}.

\begin{cor}\label{c:Kn} There exist constants $c,C>0$ such that
$c \log{\log{n}}\le \ept(K_{n}) \le C \log n$.
\end{cor}

The next lemma will be used to establish a general upper bound. 

\begin{lem} \label{lem:starlogdist} 
Let $G$ be a graph and suppose that the vertices $v_1, \dots, v_b$ each have degree at most $k$ and all are colored blue.  Then the expected number of rounds until all of their neighbors are colored blue is $O(\log{b} \log{k})$. 
\end{lem}

\begin{proof}
In Lemma \ref{lem:starlog}, we proved that if $v$ is a vertex with $k$ neighbors, then once $v$ turns blue, the expected number of rounds before all of the neighbors of $v$ are blue is at most $c \log{k}$ for some constant $c$.

Define a \emph{block} as a consecutive sequence of $2c \log{k}$ rounds. By Markov's inequality 
the probability that all of the neighbors of $v$ get colored within one block after $v$ is colored is at least $\frac 1 2$. We bound the expected number of blocks for all neighbors of the vertices $v_1, \dots, v_b$ to be successfully colored from the first time at which all of $v_1, \dots, v_b$ have been colored blue.

Let $X_{i}$ be the random variable for the number of blocks that it takes for {all the neighbors of $v_i$ to be colored blue,} 
 and define $X = \max(X_{i}: 1 \leq i \leq b)$. If $F(x) = \min(\pr(X_{i} \leq x): 1 \leq i \leq b)$, then observe that $ \pr(X \leq x) \geq F(x)^{b} $. 

Note that $F(x) \geq 1-\lp\frac 1 2\rp^{\lfloor x \rfloor}$, so 
\bea \e[X] \leq \int_{0}^{\infty} \lp1-\lp1-\lp\frac 1 2\rp^{\lfloor x \rfloor}\rp^{b}\rp dx 
&=& \sum_{n = 0}^{\infty} \lp1-\lp 1-\lp\frac 1 2\rp^{n}\rp^{b}\rp \\
&\leq& \sum_{n = 0}^{\infty} \min \lp1, b \lp\frac 1 2\rp^{n}\rp \leq \lf \log_{2}b \rf + 3.\eea Thus if the vertices $v_1, \dots, v_b$  all are colored blue, then  the expected number of rounds for all neighbors of $v_1, \dots, v_b$ to get colored is $O(\log{b} \log{k})$.
\end{proof}

Since a vertex at a distance $r$ from the one initially blue vertex cannot be reached in fewer than $r$ rounds, it is natural to develop general bounds that apply to all graphs in terms of both the radius, $\rad(G)$, and the order of $G$.   

\begin{thm}\label{t:allgraphbds}
For all connected graphs $G$ of order $n$, \[\rad(G) \leq \ept(G) = O(\rad(G) (\log n)^2)\]
and the lower bound is asymptotically tight.  \end{thm}

\begin{proof}
The lower bound of $\rad(G)$ is immediate because the vertices colored in round $i$ can be distance at most $i$ from the one vertex that was colored blue initially.  The path and cycle show that the lower bound is asymptotically tight (see Propositions \ref{prop:ept-path} and \ref{prop:ept-cycle}).

For the upper bound, initially color a center vertex of $G$ blue. At an arbitrary step of the coloring process, suppose that there are $b \leq n$ blue vertices $v_1, \dots, v_b$ that have at least one white neighbor. In Lemma \ref{lem:starlogdist}, we proved that if $v_{1}, \dots, v_{b}$ are vertices each with at most $k$ neighbors, then there exists a constant $c$ such that once $v_{1}, \dots, v_{b}$ are all blue, the expected number of rounds before all of their neighbors are blue is at most $c \log{k} \log{b}$.

 Thus after the round during which the last of $v_1, \dots, v_b$ is colored blue, the expected number of rounds for all neighbors of $v_1, \dots, v_b$ to get colored blue is $O((\log{n})^{2})$. Since all vertices in $G$ are within distance $\rad(G)$ of the initial blue vertex, the expected number of rounds until every vertex in $G$ is blue is $O(\rad(G) (\log{n})^{2})$.
\end{proof}

We do not have any examples showing the upper bound in Theorem \ref{t:allgraphbds} is tight.

\begin{cor}\label{cor:randept}
For fixed $0 < p < 1$, with high probability $\ept(G(n, p)) = O((\log n)^2)$.
\end{cor}

\begin{proof}
For fixed $0 < p < 1$, with high probability the random graph $G(n, p)$ has diameter $2$, so with high probability $G(n, p)$ has radius at most $2$. Thus the result follows from Theorem \ref{t:allgraphbds}.
\end{proof}

In a tree each vertex is a cut-set, which means that in PSD zero forcing (introduced in \cite{smallparam} and defined below) a blue vertex can force each white neighbor (deterministically) in each time step.  This resembles probabilistic zero forcing because in the latter each blue vertex fires at each white neighbor. 
The \textit{PSD color change rule} consists of coloring $w_i\in W_i$ blue when $w_i$ is the only white neighbor in $G[W_i \cup B]$ of a blue vertex $v$, where $B$ is the set of blue vertices and $W_1,...,W_k$ are the sets of white vertices corresponding to the connected components of $G-B$.  Other terms such as PSD zero forcing set are  defined analogously to those for zero forcing.  Starting with a PSD zero forcing set of blue vertices $Z$, the number of time steps required for this process to color all vertices is the \textit{PSD propagation time} of  $Z$, denoted by $\ptp(G,Z)$ \cite{PSDpropTime}.   

\begin{rem}\label{obs:treeprop}
Let $T$ be a tree and   $Z$ be a nonempty set of vertices of $T$.  Since each vertex of $T$ is a cut-set,  the probability of a blue vertex $v$ forcing a white neighbor $w$ is one using deterministic PSD zero forcing, and at most one using probabilistic zero forcing.  Thus the  propagation time for probabilistic zero forcing cannot be less than PSD propagation time. 
 This implies $\ept(T, Z) \geq \ptp(T, Z)$ for every tree $T$ and  nonempty set $Z$ of vertices in $T$.

Similarly,  for every cycle $C$ and set $Z$ of vertices in $C$ with $|Z| \geq 2$, $\ptp(C, Z) \leq \ept(C, Z)$, because any two blue vertices allow deterministic PSD zero forcing to begin.  
\end{rem}

A {\em spider} is a tree with exactly one vertex of degree at least three, which is called the  {\em body vertex}. The {\em legs} are the paths that result from deleting the body vertex.  The number of legs is the degree of the body vertex.

\begin{prop}\label{prop:spiderep}
Let  $G$ be a spider with $k$ legs.  Then $\ept(G) = \rad(G) + O(\log{k})$.
\end{prop}

\begin{proof}
By Theorem \ref{t:allgraphbds}, $\ept(G) \ge \rad(G)$. For the upper bound, we initially color a center $v$ of $G$ blue. Let $u$ be the body vertex of $G$. If $v \neq u$, then the expected time of first force is $\frac 4 3$.  After the first force the process becomes deterministic until $u$ is colored blue. By Lemma \ref{lem:starlog}, the number of rounds after $u$ is colored blue for all of $u$'s neighbors to get colored blue is $O(\log{k})$. Then the process becomes deterministic until the graph is all blue. This proves the upper bound since all vertices of $G$ are within $\rad(G)$ of $v$.
\end{proof}

 Recall that a {\em full $k$-ary tree of height $h$}, denoted by $T_{k,h}$, is constructed from a root by performing $h$ steps in which $k$ leaves are appended to each vertex of degree at most one. Observe that the order  of $T_{k,h}$ is $n = \frac{k^{h+1}-1}{k-1}$, so $h=\log_{k}((k-1)n+1)-1$.  

\begin{prop}\label{prop:k-ary}
Let  $T_{k,h}$ be  a full $k$-ary tree of order $n$.  Then $\log_{k}((k-1)n+1)-1\le \ept(T_{k,h}) = O((\log{n})^2)$, where the constant in the upper bound depends on $k$.
\end{prop}

\begin{proof}
The lower bound follows from Theorem  \ref{t:allgraphbds} and $\rad(T_{k,h})=h=\log_{k}((k-1)n+1)-1$.   For the upper bound, we initially color the root vertex $v$ of $G$ blue. The expected number of rounds for all of the neighbors of $v$ to be colored blue is $O(\log{k})$ by Lemma \ref{lem:starlog}.

Suppose that at some stage of the coloring process, all of the $k^{t}$ vertices $v_1, \dots, v_{k^{t}}$ in level $t$ of the $k$-ary tree have been colored blue. By Lemma \ref{lem:starlogdist}, the expected number of rounds for all neighbors of $v_1, \dots, v_{k^{t}}$ to get colored after $v_1, \dots, v_{k^{t}}$ have been colored is $O(t)$, where the constant in the bound depends on $k$.

Since $T_{k,h}$ has $h=\Theta(\log{n})$ levels (where the constants in the bound depend on $k$), the expected number of rounds for every vertex in $G$ to be colored is $\sum_{t = 1}^{O(\log{n})} O(t) = O((\log{n})^2)$.
\end{proof}

In zero forcing adding an edge can raise or lower the number of vertices in a minimum zero forcing set, and raising the number of vertices usually lowers propagation time and vice versa.  In theory adding an edge could either raise or lower the expected propagation time (raise it by raising a degree of a vertex, lower it by providing a shorter route between vertices.  In practice adding an edge  never seems to raise expected propagation time, but our data is limited.

\begin{quest}\label{q:mon} For all graphs $G$ and $H$ on the same set of vertices with $H$ a subgraph of $G$, is $\ept(G) \leq \ept(H)$?
\end{quest}


\section{Discussion of Kang and Yi's $P_A(G)$}\label{s:PAG}
In this section we discuss  Kang and Yi's parameter $P_A(G)$ defined in \cite{KY13} and present a counterexample to one of its claimed properties.

A graph $G$ together with an assignment of one of the colors blue and white to each vertex of $G$ is called a {\em colored graph}.  We use $G_B$ to denote the colored graph with underlying graph $G$ and set of blue vertices $B$.  For $B\subseteq V(G)$, $S_B^k$ denotes the set of colored graphs  that are possible (i.e., have positive probability) after the $k$th round starting with $G_B$; note that $S_B^0=\{G_B\}$.    For $R^k\subseteq S_B^k$, $\Pk k(R^k)$ is the probability that after round $k$ the result is one of the colored graphs in $R^k$. Let $T_B^k=\{G_Z\in S_B^k: Z\mbox{ is a  zero forcing set for }G\}$.  Then $P_B(G)=\Pk{k_0}(T_B^{k_0})$ where $k_0$ is the least $k$ that $T_B^k\ne\emptyset$ \cite{KY13} (and $P_\emptyset(G)=0$).  


In \cite[page 13]{KY13}, Kang and Yi claim the following three properties are clear for $P_B(G)$:
\ben
\item $P_\emptyset(G)=0$.
\item If $Z$ is a zero forcing set for $G$, then $P_Z(G)=1$.
\item If $A\subseteq B\subseteq V(G)$, then $P_A(G)\le P_B(G)$.
\een
The first of these properties is by definition (this was not explicit in the definition in \cite{KY13} but is clearly what is intended), and the second follows from the fact that $T_Z^0=\{G_Z\}$ since $Z$ is a zero forcing set,  and thus $\Pk 0 (T_Z^0)=1$.  However, the third property, $A\subseteq B\Rightarrow P_A(G)\le P_B(G)$, is not true.  The problem is that the definition depends on the round in which it is first possible to have a zero forcing set colored blue.  With a larger $B$, this may occur in an earlier round but with lower probability.  This is illustrated in the next example.

\begin{figure}[!ht]
\begin{center}
\scalebox{.8}{\includegraphics{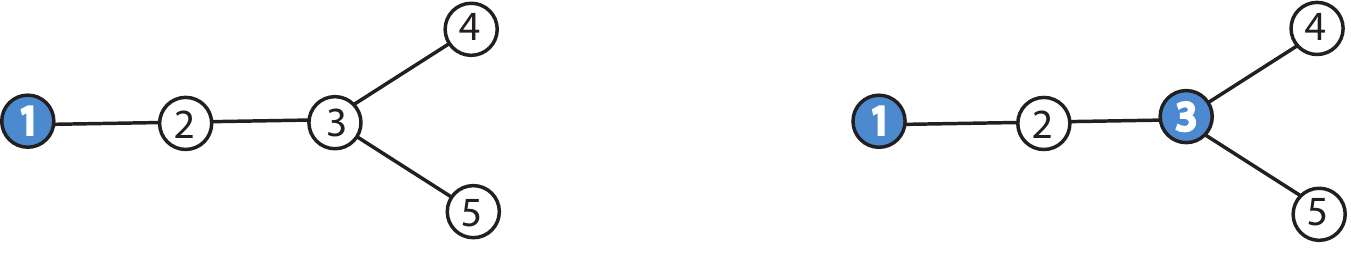}}\\
$\null$\qquad\ \ $G_A$ \qquad\qquad \qquad\qquad\qquad\qquad\qquad\qquad\ $G_B$ 
\caption{Two colorings of a graph $G$ with $A\subset B$ and $P_A(G)>P_B(G)$.}\label{fig:counterex}\vspace{-15pt}
 \end{center}
\end{figure}

\begin{ex}\label{ex:counterex} Let $G$ be the graph shown  in Figure \ref{fig:counterex} with two colorings  $A=\{1\}$ and $B=\{1,3\}$.  
Any zero forcing set for $G$ must contain at least one of 4 and 5, and  this is sufficient to guarantee a zero forcing set is blue given that vertex 1 is blue. 

For $G_A$, it takes at least three rounds to reach vertex 4 or 5 and it is possible to color  4 or 5 blue in the third round.   Thus $P_A(G)=\Pk 3 (T_A^3)$ where $T_A^3$ is the set of  all colored graphs attainable from $G_A$ in three rounds that have at least one of 4 or 5 blue. Vertex  2 is forced in the first round and vertex 3 is colored blue in the second round, so $T_A^3=\{G_{\{1,2,3,4\}},G_{\{1,2,3,5\}},G_{\{1,2,3,4,5\}}\}$.  
  The probability $3\to 4$ (or $3\to 5$) in the third round is $\frac 2 3$, so the probability of at least one of 4 and 5 being colored blue is $1-\lp\frac 1 3 \rp\lp\frac 1 3\rp=\frac 8 9=\Pk 3 (T_A^3)=P_A(G)$.    

For $G_B$, it is possible to color at least one of 4 or 5 in round one.  The probability of $3\to 4$ (or $3\to 5$)  is $\frac 1 3$, so the probability of at least one being forced in round one is 
$1-\lp\frac 2 3 \rp\lp\frac 2 3\rp=\frac 5 9=\Pk 1 (T_B^1)=P_B(G)$.

Thus
\[ P_A(G)=\frac 8 9>\frac 5 9 = P_B(G).\]
\end{ex}

\section{$\ell$-round probability}\label{s:ell-round}

While we believe that expected propagation time is the most interesting parameter associated with probabilistic zero forcing, in this section we define the $\ell$-round probability, which has the first and third properties desired by Kang and Yi  and a modified form of the second (see Proposition \ref{p:KYwant} below). We find these three properties intuitive for $\ell$-round probability, but the proof of the third is surprisingly subtle.  
Given a connected graph $G$, a set $B$ of blue vertices, and a positive integer $\ell$, the {\em $\ell$-round probability of $B$},
$\Pk\ell_B(G)$, is the probability that all vertices of $G$ are blue after $\ell$ rounds of probabilistic zero forcing starting with exactly the vertices of $B$ blue.  The {\em  $\ell$-round probability of $G$},  
$\Pk\ell(G)$, is the maximum of $\Pk\ell_B(G)$ over all one vertex sets  $B$.  Note that Aazami \cite{Aaz08, Aaz10} introduced the study of $\ell$-round zero forcing (an $\ell$-round zero forcing set must be able to color the entire graph blue in at most $\ell$ time steps).

\begin{prop}\label{p:KYwant}
Let $G$ be a graph and $\ell$ be a positive integer.  Then\vspace{-5pt}
\ben[{\rm (i)}]
\item\label{KY1} $\Pk\ell_\emptyset(G)=0$.\vspace{-5pt}
\item\label{KY2} If $Z$ is a zero forcing set for $G$ and $\ell\ge \pt(G,Z)$, then $\Pk\ell_Z(G)=1$.\vspace{-5pt}
\item\label{KY3} If $A\subseteq B\subseteq V(G)$, then $\Pk\ell_A(G)\le \Pk\ell_B(G)$.
\een
\end{prop}
\begin{proof}
\eqref{KY1}: Since only blue vertices can force, $\Pk\ell_\emptyset(G)=0$.  
For \eqref{KY2}, let $Z$ be  a zero forcing set.  Using only zero forcing  will color the entire graph blue in  $\pt(G,Z)$ time steps (rounds).  Since probabilistic zero forcing is not slower, $\Pk\ell_Z(G)=1$.

For \eqref{KY3}, we prove a stronger result. For any $C, A \subset V(G)$, define $\Pk\ell_A(C)$ to be the probability that all vertices of $C$ are blue after $\ell$ rounds of probabilistic zero forcing starting with exactly the vertices of $A$ blue. We prove that $\Pk\ell_A(C)\le \Pk\ell_B(C)$ for any $C \subset V(G)$ and $A\subseteq B\subseteq V(G)$.  Choosing $C=V(G)$ then establishes \eqref{KY3}.

Suppose that $A\subseteq B\subseteq V(G)$. We treat the event space for the probabilistic zero forcing process through the $\ell^{th}$ round as a multidimensional unit cube, where each possible event  corresponds to a dimension in the event space. For each ordered pair $(u,v)$ such that $u$ and $v$ are adjacent and each  $i=1,\dots, \ell$, we define the event  \emph{vertex $u$ colors vertex $v$ in round $i$}  to mean that $v$ is blue after round $i-1$ or $u\to v$ in round $i$ (it does not matter if other vertices also force $v$ in round $i$). This differs slightly from probabilistic zero forcing in that $u$ cannot force $v$ if $v$ is already blue. 
However, this modification results in the same vertices being blue and white after each round. We can define an event space for each of  $A$ and $B$.

For the event space of $A$, we define a hierarchy of partitions $\mathcal{R}_{A, i}$ of the space for $0 \leq i \leq \ell$. First $\mathcal{R}_{A, 0}$ is defined as the whole cube. Given $\mathcal{R}_{A, i}$, we define $\mathcal{R}_{A, i+1}$ to consist of the multidimensional rectangles obtained from those in $\mathcal{R}_{A, i}$ by cutting each rectangle on the dimensions corresponding to the possible events for the $(i+1)^{st}$ round according to their probabilities of occurrence. Each rectangle will also be assigned a subset of vertices as a label; $\mathcal{R}_{A, 0}$ is labeled with the subset $A$. For each rectangle $R$ in $\mathcal{R}_{A, i}$, the rectangles obtained from $R$ in $\mathcal{R}_{A, i+1}$ will have labels that contain $R$'s label as a subset. 
Next we describe how to generate rectangles in $\mathcal{R}_{A, i+1}$ from the rectangles and labels of $\mathcal{R}_{A, i}$.

 If $v$ is blue after round $i$ in $R$, then the dimension corresponding to each event of the form \emph{vertex $u$ colors vertex $v$ in round $i+1$} in $R$ will be cut into lengths of $1$ (probability of success) and $0$ (probability of failure). 
If vertices $u$ and $v$ are both not blue after round $i$ in $R$, then the dimension corresponding to each event of the form \emph{vertex $u$ colors vertex $v$ in round $i+1$} in $R$ will be cut into lengths of $0$ (probability of success) and $1$ (probability of failure). If $u$ is blue, has $b$ blue neighbors after round $i$ in $R$ (i.e., $b$ neighbors of $u$ are in the label for $R$), and $\deg u=d > b$, then the dimension corresponding to each event of the form \emph{vertex $u$ colors vertex $v$ in round $i+1$} in $R$ for white neighbor $v$ will be cut into lengths of $\frac{b+1}{d}$ (probability of success) and $\frac{d-b-1}{d}$ (probability of failure). 

For each of the rectangles in $\mathcal{R}_{A, i}$ for $i \geq 1$, we label the rectangle with the subset of vertices that are blue either because they are in $A$ or as a result of a successful event in the rectangle. Note that any dimensions of the rectangle corresponding to failed fires do not contribute any elements to the subset label for the rectangle. Observe that the volume of each rectangle is the probability that all of the events in the rectangle occur in the first $i$ rounds, and that $\Pk{i}_A(C)$ is the sum of the volumes of the rectangles in $\mathcal{R}_{A, i}$ whose label sets contain $C$ as a subset.

For each $i$, we now transform $\mathcal{R}_{A, i}$ into a set of rectangles for the event space of $B$ up to turn $i$ by cutting a subset of the rectangles in each dimension and modifying their subset labels. The modifications are performed in order of rounds. The only modification to $\mathcal{R}_{A, 0}$ is its subset label, which is changed from $A$ to $B$. For each rectangle $R \in \mathcal{R}_{A, i}$, the rectangles obtained from modifying $R$ will be called the \emph{descendants} of $R$. The transformation will have the property that $\Pk\ell_B(C)$ is the sum of the volumes of the descendants from $\mathcal{R}_{A, \ell}$ whose labels contain $C$ as a subset. Also note that the transformation will only cut dimensions of rectangles corresponding to failed fires. Since dimensions of rectangles corresponding to failed fires do not contribute any elements to subset labels, the transformation will only add elements to the subset labels of all descendants (or leave the subset labels unchanged), so we will obtain  $\Pk\ell_A(C)\le \Pk\ell_B(C)$.

Since $A \subseteq B$, the transformation adds a nonnegative number of blue vertices at the beginning of the process, and propagates their effects through the first $i$ rounds. Fix $R \in \mathcal{R}_{A, i+1}$, and suppose that $R$ was constructed from $Q \in \mathcal{R}_{A, i}$. For any vertex $u$ that is blue as a result of the first $i$ rounds in $R$ (and $Q$), suppose that $u$ has $b$ blue neighbors after round $i$ in $R$ (and $Q$) and $\deg u =d > b$. Thus the sidelength of $R$ corresponding to the event \emph{vertex $u$ colors vertex $v$ in round $i+1$} for white neighbor $v$ is either $\frac{b+1}{d}$ or $\frac{d-b-1}{d}$ depending on whether this dimension is a success or failure in $R$.

To construct each descendant of $R$, we cut $R$ in the dimensions corresponding to the events of the first $i$ rounds to form the descendants of $Q$ when restricted to those dimensions, and we add the same vertices to the subset labels of the new rectangles that were added in the descendants of $Q$. For any descendant $Q'$ of $Q$, if adding the vertices of $B-A$ causes $v$ to be colored in the first $i$ rounds in $Q'$, then the probabilities of the event \emph{vertex $u$ colors vertex $v$ in round $i+1$} are $1$ (success) and $0$ (failure). If $R$ has success for the event \emph{vertex $u$ colors vertex $v$ in round $i+1$}, then the descendants of $R$ in this dimension have the same sidelength as $R$ and gain no new vertices in the label from this dimension, since $R$ already had $v$ in the label. If $R$ has failure for the event, then we add $v$ to the subset label of the descendant of $R$ obtained from $Q'$ and make the same sidelength for the dimension as $R$.

If adding the vertices of $B-A$ does not cause $v$ to be colored in the first $i$ rounds in $Q'$ but does introduce $s$ new blue neighbors of $u$ within the first $i$ rounds, then the probabilities for the event \emph{vertex $u$ colors vertex $v$ in round $i+1$} for white neighbor $v$ are $\frac{b+s+1}{d}$ (success) and $\frac{d-b-s-1}{d}$ (failure).  If $R$ has success for the event \emph{vertex $u$ colors vertex $v$ in round $i+1$}, again the descendants of $R$ in this dimension have the same sidelength as $R$ and gain no new vertices in the label from this dimension. If $R$ has failure for the event, then the descendants of $R$ obtained from $Q'$ have sides of length $\frac{s}{d}$ and $\frac{d-b-s-1}{d}$ in the dimension (as a result of splitting the side of length $\frac{d-b-1}{d}$ in $R$) and $v$ is added to the subset label of the descendant with sidelength $\frac{s}{d}$. 

If adding the blue vertices of $B-A$ causes $u$ to be colored in the first $i$ rounds in $Q'$ (and $u$ was not colored in $R$), and if $u$ has $b$ blue neighbors after round $i$ in $Q'$ and $\deg u=d > b$, then the probabilities for the event \emph{vertex $u$ colors vertex $v$ in round $i+1$} for white neighbor $v$ are $\frac{b+1}{d}$ (success) and $\frac{d-b-1}{d}$ (failure). Since $R$ has failure for the event and length $1$, then the descendants of $R$ obtained from $Q'$ have sides of length $\frac{b+1}{d}$ and $\frac{d-b-1}{d}$ in the dimension (as a result of splitting the side of length $1$ in $R$) and $v$ is added to the subset label of the descendant with sidelength $\frac{b+1}{d}$. 

Observe that all of the alterations have been on dimensions corresponding to failed fires in $R$, and failed fires did not contribute any elements to the subset labels for $A$. Thus for each rectangle $R \in \mathcal{R}_{A, \ell}$, the alterations have produced a set of rectangles with the same volume as $R$ that each have subset labels which contain all of the vertices in the subset label of $R$. Thus the sum of the volumes of the descendants of $\mathcal{R}_{A, \ell}$ whose label subsets contain all of the elements of $C$ is at least the sum of the volumes of the rectangles in $\mathcal{R}_{A, \ell}$ whose label sets contain $C$ as a subset. This implies  $\Pk\ell_A(C)\le \Pk\ell_B(C)$. 
\end{proof}

Property \eqref{KY2} of Proposition \ref{p:KYwant} has the added stipulation that $\ell\ge\pt(G,Z)$, which is necessary  and seems reasonable given the definition of $\ell$-round probability.  In the definition of $P_A(G)$, the probability is measured after the first round in  which a zero forcing set can be colored blue, and as shown in Section \ref{s:PAG} this is incompatible with the third property.

\begin{prop}\label{p:ell-round-cycle} For a cycle of order $n > 2$, $\Pk\ell(C_{n}) = 0$ for $\ell < \lf\frac n 2\rf$, and for $\ell\ge \lf\frac n 2\rf$,  
\[\Pk\ell(C_{n}) =  \begin{cases} 
1-(\frac 1 4)^{\ell- n/2+1} & \mbox{if $n$ is even}\\
1 -\frac 3 4 \lp\frac 1 4\rp^{\ell-(n-1)/2} &  \mbox{if $n$ is odd}
\end{cases}.\]
\end{prop}

\begin{proof}
The probability that the first force occurs in the $k^{th}$ round is $\lp\frac 1 4\rp^{k-1}\lp\frac 3 4\rp$. The probability that two forces occur in the first round that has a force is $\frac 1 3$, and the probability that only one force occurs on the first round that has a force is $\frac 2 3$.

First assume $n$ is even.  Then there will be $\frac n 2 - 1$ rounds after the first force, regardless of how many forces occur on the first round  that has a force (call this round $k$). Thus,  the process takes $t=\frac n 2-1+k$ rounds with probability $\lp\frac 1 4\rp^{k-1}\lp\frac 3 4\rp=\frac 3 4\lp\frac 1 4\rp^{t-n/2}$. This implies that $\Pk\ell(C_{n}) = 0$ for $\ell < \frac n 2$ and\vspace{-5pt} \[\Pk\ell(C_{n}) =1- \sum_{t=\ell+1}^\infty \frac 3 4\lp\frac 1 4\rp^{t-n/2}=1-\lp\frac 1 4\rp^{\ell-n/2+1}\vspace{-5pt}\] for every $\ell \geq \frac n 2.$

Now assume $n$ is odd. Then there will be $\frac{n-1}2$ rounds after the first force if there is only one force on the first round that has a force, and there will be $\frac{n-3}2$ rounds after the first force if there are two forces on the first round that has a force.    The probability of two forces in the first round is $\frac 1 3\lp\frac 3 4\rp=\frac 1 4$, and $t=1+\frac{n-3} 2=\frac{n-1} 2$ rounds are needed to color all vertices blue. For each $t\ge  \frac{n-1} 2+1$,  there are two ways to achieve the last vertex turning blue in round $t$:  Only one force happens in round $t-\frac{n-1}2$ (and no forces earlier) with probability $\frac 2 3 \lp\frac 1 4\rp^{t-(n-1)/2-1}\lp\frac 3 4\rp$.  Two forces happen in round $t-\frac{n-1}2+1$ (and no forces earlier) with probability $\frac 1 3 \lp\frac 1 4\rp^{t-(n-1)/2}\lp\frac 3 4\rp$.  So the probability of the last vertex turning blue in round $t\ge  \frac{n-1} 2+1$ is $\frac 9 {16}  \lp\frac 1 4\rp^{t-(n+1)/2}$.  
This implies that $\Pk\ell(C_{n}) = 0$ for $\ell < \frac{n-1}2$ and \vspace{-5pt}
\[\Pk \ell(C_{n}) = \frac 1 4+\frac 3 4 \lp1-\lp\frac 1 4\rp^{\ell-(n-1)/2}\rp=1-\frac 3 4 \lp\frac 1 4\rp^{\ell-(n-1)/2}\vspace{-5pt}\] 
for every $\ell \geq \frac{n-1}2$. 
\end{proof}

\begin{prop}\label{p:ell-round-path}
For a path of order $n > 2$, $\Pk\ell(P_{n}) = 0$ for $\ell < \lf\frac n 2\rf$, and for $\ell\ge \lf\frac n 2\rf$,  \vspace{-5pt}
\[\Pk\ell(P_{n}) =  \begin{cases} 
1-\frac 1 2 \lp\frac 1 4\rp^{\ell-n/2} & \mbox{if $n$ is even}\\
1 -\frac 3 4 \lp\frac 1 4\rp^{\ell-(n-1)/2} &  \mbox{if $n$ is odd}
\end{cases}.\vspace{-5pt}\]
\end{prop}

\begin{proof} The probability of  the first force occurring on the $k^{th}$ round and the probabilities of one or two forces in the round with the first force are the same as for a cycle.  If $n$ is odd, then the situation is the same as for a cycle, so  $\Pk\ell(P_{n}) = 0$ for $\ell < \frac{n-1}2$ and $P^{\ell}(P_{n}) = 1 -\frac 3 4 \lp\frac 1 4\rp^{\ell-(n-1)/2}$ for every $\ell \geq \frac{n-1}2$. 


If $n$ is even, then as in the proof of Proposition \ref{prop:ept-path}, we need to distinguish whether the neighbor in the longer direction is forced: There will be $\frac n 2 - 1$ rounds after the first force with probability $\frac 2 3$, and there will be $\frac n 2$ rounds after the first force with probability $\frac 1 3$. Thus if $n$ is even, then the process takes $t=\frac n 2$ rounds with probability $\lp\frac 3 4\rp\lp\frac 2 3\rp=\frac 1 2$  and $t\ge \frac n 2 +1$ rounds with probability $\lp\frac 2 3 \rp\lp\frac 1 4 \rp^{t-n/2}\lp\frac 3 4\rp + \lp\frac 1 3 \rp\lp\frac 1 4 \rp^{t-n/2-1}\lp\frac 3 4\rp=\lp\frac 3 8\rp \lp\frac 1 4\rp^{t-n/2-1}$. This implies that $\Pk\ell(P_{n}) = 0$ for $\ell < \frac n 2$ and $\Pk\ell(P_{n}) = \frac 1 2+\frac 1 2 (1-\lp\frac 1 4 \rp)^{\ell-n/2})$ for every $\ell \geq \frac n 2$.
\end{proof}

The next lemma is an application of Markov's inequality.

\begin{lem}\label{epttoell}
If $G$ is a connected graph and $\ell > \ept(G)$, then $\Pk\ell(G) \geq 1-\frac{\ept(G)}\ell$.
\end{lem}

We can also obtain numerous corollaries about $\Pk\ell$ from prior results about expected propagation time and the next lemma, which  follows from the previous one.

\begin{lem}\label{epttoell0}
For all connected graphs $G$, $\ept(G)=O(f(n))$ implies $\Pk\ell(G) = 1-o(1)$ for $\ell = \omega(f(n))$.
\end{lem}

\begin{cor}\label{cor:pstar}
$\Pk\ell(G) = 1-o(1)$ for $\ell = \omega(\log{n})$ for every graph $G$ with a universal vertex.
\end{cor}


\begin{cor}
$\Pk\ell(K_{1, n}) = 1-o(1)$ for $\ell = \omega(\log{n})$ and $\Pk\ell(K_{1, n}) = o(1)$ for $\ell = o(\log{n})$.
\end{cor}

\begin{proof}
The first bound follows from Corollary \ref{cor:pstar}, while the second bound is proved by the method used in the proof of Theorem \ref{starlower}. 
\end{proof}

\begin{cor}
$\Pk\ell(K_{n}) = 1-o(1)$ for $\ell = \omega(\log{n})$ and $\Pk\ell(K_{n}) = o(1)$ for $\ell = o(\log{\log{n}})$.
\end{cor}

\begin{proof}
The first bound follows from Corollary \ref{cor:pstar}, while the second bound is proved by the method used in the proof of Proposition \ref{klower}.
\end{proof}

\begin{cor}
$\Pk\ell(S) = 1-o(1)$ for $\ell = \omega(\rad(S) + \log{k})$ for every spider $S$ with $k$ legs.
\end{cor}


\begin{cor}
$\Pk\ell(T) = 1-o(1)$ for $\ell = \omega((\log{n})^2)$ for every full $k$-ary tree $T$.
\end{cor}


\begin{cor}
With high probability, $\Pk\ell(G(n, p)) = 1-o(1)$ for $\ell = \omega( (\log{n})^2)$ for fixed $0 < p < 1$.
\end{cor}


\begin{cor}
$\Pk\ell(G) = 0$ for $\ell < rad(G)$ and $\Pk\ell(G) = 1-o(1)$ for all $\ell = \omega(\rad(G) (\log{n})^2)$ for every connected graph $G$.
\end{cor}

\begin{proof}
The first statement is true since the number of steps in the coloring process cannot be less than $\rad(G)$ if we start with only one blue vertex. The second statement follows from  Lemma \ref{epttoell0} and Theorem \ref{t:allgraphbds}.
\end{proof}


\section{Confidence propagation time}\label{s:alpha}

Define $\ptpf(G,Z,\alpha)$ to be the least number of rounds $t$ such that the probability that all the vertices are blue after round  $t$ is greater than or equal to $\alpha$, assuming that the vertices in $Z$ are colored initially.  This can be thought of as the the time at which you have $alpha$-confidence that the graph is all blue when staring with $Z$, and is called the {\em $\alpha$-confidence propagation time}.
Define $\displaystyle \ptpf(G,\alpha)=\min_{v\in V(G)} \ptpf(G,\{v\},\alpha)$.

Confidence propagation time can be determined immediately from $\ell$-round probability when this is known, as for cycles and paths (Propositions \ref{p:ell-round-cycle} and \ref{p:ell-round-path}).

\begin{cor}
For a cycle of order $n$, 
\[\ptpf(C_{n}, \alpha)  = \begin{cases} 
\frac{n}{2}+\ell & \mbox{if $n$ is even and }1-(\frac{1}{4})^{\ell} < \alpha < 1-(\frac{1}{4})^{\ell+1}\\
\frac{n-1}{2} &  \mbox{if $n$ is odd and } 0 < \alpha \leq \frac 1 4\\
\frac{n-1}{2}+\ell &  \mbox{if $n$ is odd and } 1-\frac{3}{4}(\frac{1}{4})^{\ell-1} < \alpha \leq 1-\frac{3}{4}(\frac{1}{4})^{\ell}
\end{cases}.\]
\end{cor}

\begin{cor}
For a path of order $n$, 
\[\ptpf(P_{n}, \alpha)  = \begin{cases} 
\frac{n}{2} &  \mbox{if $n$ is even and } 0 < \alpha \leq \frac 1 2\\
\frac{n}{2}+\ell &  \mbox{if $n$ is even and } 1-\frac{1}{2}(\frac{1}{4})^{\ell-1} < \alpha \leq 1-\frac{1}{2}(\frac{1}{4})^{\ell}\\
\frac{n-1}{2} &  \mbox{if $n$ is odd and } 0 < \alpha \leq \frac 1 4\\
\frac{n-1}{2}+\ell &  \mbox{if $n$ is odd and } 1-\frac{3}{4}(\frac{1}{4})^{\ell-1} < \alpha \leq 1-\frac{3}{4}(\frac{1}{4})^{\ell}
\end{cases}.\]
\end{cor}

The next lemma is analogous to Lemma \ref{epttoell} in the last section.

\begin{lem}\label{epttoalpha}
If $G$ is a connected graph, then $\ptpf(G,\alpha) \leq \frac{\ept(G)}{1-\alpha}$.
\end{lem}

\begin{proof}
By Markov's inequality, the probability that $G$ is not all blue by time $T$ is at most $\frac{\ept(G)}{T}$. When $T = \frac{\ept(G)}{1-\alpha}$, this probability is at most $1-\alpha$.
\end{proof}

The upper bounds for the next few corollaries follow from Lemma \ref{epttoalpha}. The proofs for the two lower bounds use the same method as in the proofs of Theorem \ref{starlower} and Proposition \ref{klower}.

\begin{cor}\label{cor:alphastar}
For every constant $0 < \alpha < 1$, $\ptpf(G,\alpha) =O(\log{n})$ for every graph $G$ with a universal vertex.
\end{cor}


\begin{cor}
For every constant $0 < \alpha < 1$, $\ptpf(K_{1, n}, \alpha) = \Theta(\log{n})$.
\end{cor}


\begin{cor}
For every constant $0 < \alpha < 1$, $\ptpf(K_{n}, \alpha) = O(\log{n})$ and $\ptpf(K_{n}, \alpha) = \Omega(\log{\log{n}})$.
\end{cor}




\begin{cor}
For every constant $0 < \alpha < 1$, $\ptpf(T,\alpha) = O((\log{n})^2)$ for every full $k$-ary tree $T$, where the constant in the bound depends on $k$.
\end{cor}


\begin{cor}
For every constant $0 < \alpha < 1$, $\rad(G) \leq \ptpf(G,\alpha) = O(\rad(G) (\log{n})^2)$ for every connected graph $G$.
\end{cor}


\begin{cor}
For every constant $0 < \alpha < 1$, with high probability $\ptpf(G(n, p), \alpha) = O((\log{n})^2)$ for all fixed $0 < p < 1$.
\end{cor}


For spiders, we obtain a tighter bound than what is given by Lemma \ref{epttoalpha}. This follows from a proof very similar to Proposition \ref{prop:spiderep}, using Markov's inequality on the sum of the times for the first force to occur and for all neighbors of the body vertex to be colored after the body vertex is colored.

\begin{prop}
If $G$ is a spider with $k$ legs, then $\ptpf(G, \alpha) = \rad(G) + O(\log{k})$, where the constant in the $O(\log{k})$ depends on $\alpha$.
\end{prop}





 \section{Probabilistic throttling}\label{s:throt}

Many results about expected propagation time can be applied to obtain results about  probabilistic throttling,  $\thpf(G,Z)=|Z|+\ept(G,Z)$, and the probabilistic throttling number  $ \thpf(G)=\min\{\thpf(G,Z)\}$.
The next result follows from Lemma \ref{lem:starlogdist} by choosing a dominating set, with $\gamma (G)$ denoting the domination number and $\Delta(G)$ denoting the maximum degree.

\begin{cor}\label{cor:domnumber}
For any connected graph $G$ of order $n$,  $\thpf(G) \le \gamma(G)+O(\log{\gamma(G)}\log{\Delta(G)})$. 
\end{cor}

\begin{cor}
With high probability, $\thpf(G(n, p)) = O(\log{n} \log{\log{n}})$ for any fixed $0 < p < 1$.
\end{cor}
\bpf With probability at least $1-o(1)$, $G(n, p)$ has domination number $O(\log{n})$. Thus by Corollary \ref{cor:domnumber}, the throttling number of $G(n, p)$ is $O(\log{n} \log{\log{n}})$ with high probability.  \epf

For a tree $T$, we use $\thp(T)$ to denote the throttling number of $T$ for PSD zero forcing. On trees, this throttling number coincides with the throttling numbers for Cops and Robbers and distance domination, which are both equal to $\min_k\{k+\rad_k(T)\}$ \cite{CRthrottle}.  By Observation \ref{obs:treeprop}, $\thp(T,Z)\le \thpf(T,Z)$ for all sets of vertices $Z$, and similarly for cycles when $Z$ contains at least two vertices.

\begin{prop}\label{c_n}
$\thpf(C_{n}) = \sqrt{2n} + O(\log{n})$ and $\thpf(P_{n}) = \sqrt{2n} + O(\log{n})$
\end{prop}

\begin{proof}
The lower bound follows from the relationship with PSD throttling and the fact that $\thp(P_n)=\thp(C_n)=\lc\sqrt{2n}-\frac 1 2\rc$ (for the cycle $n\ge 4$ is needed) \cite[Theorems 3.2, 3.3]{PSDthrottle}. 

For the upper bound, suppose that we place $k = \lc \sqrt{\frac n 2}\, \rc$ blue vertices $v_1, \dots, v_{k}$ on the graph so that every vertex in the graph is within distance $\lf \sqrt{\frac n 2} \rf$ of a blue vertex, as in the proofs of Theorems 3.2 and 3.3 in \cite{PSDthrottle}. 
Note that the expected time for all of the initial blue vertices to have a blue neighbor is $O(\log{k})=O(\log n)$ by Lemma \ref{lem:starlogdist}. After all initial blue vertices have a blue neighbor, the probabilistic zero forcing process becomes deterministic, so the remaining time to color the whole graph blue is at most $\lf \sqrt{\frac n 2} \rf$. 
\end{proof}

\begin{prop}\label{lem:thspider}
If $T$ is a spider of order $n$, then $\thpf(T) = \thp(T)+O(\log{n})$.
\end{prop}

\begin{proof}
The lower bound is immediate since $\thpf(T) \geq \thp(T)$ for any tree $T$. 
For the upper bound, choose a  set $Z$ of vertices in $T$ that achieves $\thp(G)$ and initially color all vertices in $Z$ blue.
Let $u$ denote  the body vertex. 

 The expected number of rounds before all vertices in $Z$ other than $u$ (if $u\in Z$) have a blue neighbor is $O(\log{n})$ by Lemma \ref{lem:starlogdist}. The expected number of rounds before all neighbors of $u$ are successfully forced is $O(\log{n})$ by Lemma \ref{lem:starlog}.

Thus the expected number of rounds for all vertices in $Z$ to have a blue neighbor (and for all neighbors of $u$ to be successfully forced if $u \in Z$) is $O(\log{n})$. If $u\in Z$, then the process is deterministic until $G$ is all blue and the upper bound follows. If $u\notin Z$, then the process is deterministic until $u$ is colored blue. The expected number of rounds after that point for all neighbors of $u$ to get colored blue is $O(\log{n})$. Then the process is deterministic until $G$ is all blue and the upper bound follows.
\end{proof}

\begin{thm}\label{thm:pzfth}
Among connected graphs of order $n$, the maximum possible probabilistic  throttling number is $\Omega(\sqrt{n})$ and $O(\sqrt{n} \log(n)^2)$.
\end{thm}

\begin{proof}
By Proposition  \ref{c_n}, the path and the cycle achieve $\Omega(\sqrt{n})$.  For the upper bound, let $T$ be a spanning subtree of $G$. Initially color a subset $Z$ of $\lf \sqrt{n} \rf$ vertices such that $Z$ is a $\lf \sqrt{n} \rf$-center of $T$.

At an arbitrary step of the coloring process, suppose that there are $b \leq n$ blue vertices $v_1, \dots, v_b$ that have at least $1$ white neighbor. By Lemma \ref{lem:starlogdist}, we have that after the first round during which all of $v_1, \dots, v_b$ are colored, the expected number of rounds for all neighbors of $v_1, \dots, v_b$ to get colored is $O((\log{n})^{2})$. 

Since all vertices in $T$ are within distance $\lc \sqrt{n}\, \rc$ of a vertex in $Z$, the expected number of rounds until every vertex in $T$ is blue is $O(\sqrt{n} (\log{n})^{2})$.
\end{proof}


\section{Confidence throttling}\label{s:alphath}
 In this section we define throttling for confidence propagation time. The {\em confidence throttling number} of a graph $G$ is 
\[\thr(G, \alpha)= \min_{Z\subseteq V(G)}\lp|Z|+\ptpf(G,Z,\alpha)\rp.\] 

In order to derive a general upper bound on confidence throttling for connected graphs, we use a bound on $\ptpf(G,Z,\alpha)$. The next lemma generalizes Lemma \ref{epttoalpha} and is proved the same way.
\begin{lem} \label{epttoalphath}
If $G$ is a connected graph and $Z$ is any set of vertices, then $\ptpf(G,Z,\alpha) \leq \frac{\ept(G,Z)}{1-\alpha}$.
\end{lem}

\begin{cor}
For all $0 < \alpha < 1$, the maximum possible value of $\thr(G,\alpha)$ among connected graphs of order $n$ is $O(\sqrt{n} \log(n)^2)$, where the constant depends on $\alpha$.
\end{cor}

The next result on confidence throttling for paths and cycles has nearly the same proof as Proposition \ref{c_n} from the last section. The expected time for all of the vertices in the initial blue set to have a blue neighbor is $O(\log{n})$ by Lemma \ref{lem:starlogdist}, so by Markov's inequality, it takes $O(\log{n})$ rounds for the probability to exceed $\alpha$ that all of the vertices in the initial blue set have a blue neighbor, where the constant in the $O(\log{n})$ depends on $\alpha$.

\begin{lem}
For all\, $0 < \alpha < 1$, $\thr(C_n,\alpha) = \sqrt{2n}+O(\log{n})$ and $\thr(P_n,\alpha) = \sqrt{2n}+O(\log{n})$, where the constant in the $O(\log{n})$ depends on $\alpha$.
\end{lem}

\begin{cor}
For all $0 < \alpha < 1$, the maximum possible value of $\thr(G,\alpha)$ among connected graphs of order $n$ is $\Omega(\sqrt{n})$.
\end{cor}

Like in the last section, we obtain a tighter throttling bound for spiders.  This follows from a proof very similar to Proposition \ref{lem:thspider}, using Markov's inequality on the sum of the times for every vertex in the initial blue set to have a blue neighbor and for all white neighbors of the body vertex to be forced after the body vertex is colored.

\begin{lem}
If $T$ is a spider of order $n$, then $\thr(T,\alpha) = \thp(T)+O(\log{n})$, where the constant in the $O(\log{n})$ depends on $\alpha$.
\end{lem}

\section{Concluding remarks}\label{s:conc}

In \cite{KY13} Kang and Yi provide key definitions for   probabilistic zero forcing:  the probability of a force \eqref{eq:pforce} and the concept of a round. 
Here we use these definitions to begin the study of expected propagation time, $\ell$-round probability, probabilistic throttling, and confidence propagation and throttling.  
The fact that so many results from expected propagation time can be applied to obtain results on $\ell$-round probability,  confidence propagation, and probabilistic throttling provides additional evidence that expected propagation time is a parameter worthy of study.

Many questions about expected propagation time remain, including Question \ref{q:mon} and the determination of an asymptotically tight upper bound for all graphs of order $n$.  Another question that may be of interest is how to identify a vertex $u\in V(G)$ such that $\ept(u)\le \ept(v)$ for all $v\in V(G)$.  Also, the idea in Corollary \ref{cor:domnumber} of using a dominating set could be extended to a distance dominating set. 


\end{document}